%% file: ullrich_final.tex
\documentclass[12pt]{article}

\usepackage{amsmath}
\usepackage{amsthm}
\usepackage{amssymb}
\usepackage{amsfonts}
\usepackage{bbm}
\usepackage{dsfont}
\usepackage{graphicx}
\usepackage{pstricks,pstricks-add,pst-math,pst-xkey}
\usepackage{hyperref}

\newtheorem{theo}{Theorem}
\newenvironment{theorem}{\vspace{4mm}\begin{theo}}{\end{theo}}
\newtheorem{lem}[theo]{Lemma}
\newenvironment{lemma}{\vspace{4mm}\begin{lem}}{\end{lem}}
\newtheorem{coro}[theo]{Corollary}
\newenvironment{corollary}{\vspace{4mm}\begin{coro}}{\end{coro}}
\newtheorem{rem}[theo]{Remark}
\newenvironment{remark}{\vspace{4mm}\begin{rem}\rm}{\end{rem}}
\newtheorem{prop}[theo]{Proposition}
\newenvironment{proposition}{\vspace{4mm}\begin{prop}}{\end{prop}}
\newtheorem{defi}{Definition}

\newtheoremstyle{citing}{}{}{\itshape}{}{\bfseries}{.}{ }{\thmnote{#3}}
\theoremstyle{citing}
\newtheorem{cit}{}

\newcommand{\Z}{\mathbb{Z}}
\newcommand{\N}{\mathbb{N}}
\newcommand{\R}{\mathbb{R}}
\newcommand{\ind}{\scalebox{1.3}{\raisebox{-1.5pt}{$\mathds{1}$}}}
\newcommand{\abs}[1]{\left\vert #1 \right\vert}	% absolute value / norm
\newcommand{\norm}[1]{\left\Vert #1 \right\Vert}	% absolute value / norm

\newcommand{\Var}{\mbox{\rm Var}}
\newcommand{\EE}{\mathcal{E}}
\renewcommand{\l}{\langle}
\renewcommand{\r}{\rangle}

\newcommand{\eps}{\varepsilon}
\renewcommand{\O}{\Omega}

\begin{document}

\title{Comparison of Swendsen-Wang and Heat-Bath Dynamics}
\author{Mario Ullrich\footnote{The author was supported 
																by the DFG GK 1523.}\\
	Mathematisches Institut, Universit\"at Jena\\   
	Ernst-Abbe-Platz 2, 07743 Jena, Germany\\   
	email: mario.ullrich@uni-jena.de}
	
\date{March 15, 2012}

\maketitle

\begin{abstract}
We prove that the spectral gap of the Swendsen-Wang process for 
the Potts model on graphs with bounded degree 
is bounded from below by some constant times the spectral gap 
of any single-spin dynamics.  
This implies rapid mixing for the two-dimensional Potts model 
at all temperatures above the critical one, as well as rapid mixing 
at the critical temperature for the Ising model. 
After this we introduce a modified version of the 
Swendsen-Wang algorithm for planar graphs and prove rapid mixing 
for the two-dimensional Potts models at all non-critical temperatures.
\end{abstract}

\thispagestyle{empty}

\section{Introduction}

We study the mixing properties (in terms of the spectral gap) of the 
Swendsen-Wang dynamics for the $q$-state Potts model on graphs 
with bounded degree. 
The Swendsen-Wang process is probably the most widely used algorithm 
to generate (approximate) samples from the Potts model because,
empirically, it seems to be rapid for all graphs and temperatures, 
at least for small $q$. But this is in general not true. 
It is known that the process does not mix rapidly at some critical 
temperature on the complete graph for all $q\ge3$, see Gore and 
Jerrum \cite{GJ}, 
and on rectangular subsets of the hypercubic lattice $\Z^d$, 
$d\ge2$, for $q$ sufficiently large, see the recent paper of 
Borgs, Chayes and Tetali \cite{BCT}. 
Furthermore, Li and Sokal \cite{LiS} proved that there is a lower 
bound on the mixing time in terms of the \emph{specific heat}.
Results on rapid mixing of the Swendsen-Wang dynamics are also rare. 
Rapid mixing is proven at all temperatures on trees and cycles for 
any $q$, see Cooper and Frieze \cite{CF}, on narrow grids for any 
$q$, see Cooper et al. \cite{CDFR}, 
and on the complete graph for $q=2$, see \cite{CDFR} and Long et al. 
\cite{LNP}. 
For rapid mixing results for sufficiently high or low temperatures 
see e.g. Martinelli \cite{M2}, Cooper and Frieze \cite{CF} and 
Huber \cite{Hu}.

In this article we prove that the spectral gap of the 
Swendsen-Wang dynamics on graphs with bounded degree is bounded 
from below 
by some constant times the spectral gap of the (single-site) 
heat-bath dynamics. The main result of this paper can be stated as 
follows, see Theorem~\ref{theorem:main}.

Suppose that $P$ (resp. $P_{\rm HB}$) is the transition matrix of 
the Swendsen-Wang (resp. heat-bath) dynamics, which is reversible 
with respect to the Boltzmann distribution $\pi^{G}_{\beta,q}$ of 
the $q$-state Potts model on the graph $G$ at inverse 
temperature $\beta$.
Then there exists a constant $c>0$, depending only on $q$ and 
$\beta$, such that
\[
\lambda(P) \;\ge\; c^\Delta\,\lambda(P_{\rm HB}),
\]
where $\lambda(\cdot)$ denotes the spectral gap and $\Delta$ is 
the maximal degree of the underlying graph $G$. 
In Theorem~\ref{theorem:main2}$'$ (Subsection 4.3.2) we prove this 
statement also for the case that we have only a single vertex $v^*$ 
of large degree, where $\Delta$ is replaced by the second largest 
degree and $P_{\rm HB}$ is replaced by the transition matrix of 
the heat bath chain that leaves the color at $v^*$ constant.

In particular, this implies (see Corollary~\ref{coro:2d}) that 
the Swendsen-Wang process is rapidly 
mixing above the critical temperature for the two-dimensional 
Potts model ($q\ge2$) and at the critical temperature for $q=2$.
Furthermore we get rapid mixing for graphs with maximum 
degree $\Delta$ if $p=1-e^{-\beta}<1/\Delta$. If the graph is 
planar, then $p=1-e^{-\beta}<1/(3\sqrt{(\Delta-3)})$ is 
sufficient for $\Delta\ge6$ (see Corollary~\ref{coro:planar}). 
This improves over the $1/(3(\Delta-1))$ bound of \cite{Hu}.

One may believe (and numerical results suggest) that there 
could be a sharper bound on the spectral gap 
of the Swendsen-Wang with an additional factor $N$ 
(= number of vertices of the graph) on the right hand side, 
i.e. $\lambda(P)=\Omega(N \lambda(P_{\rm HB}))$. But proving 
this does not seem to be possible with the techniques of this 
paper.

We conjecture that the spectral gap of 
Swendsen-Wang is (almost) the 
same at high and low temperatures, but we were not able to 
prove this. Nevertheless we present a modified version of the 
Swendsen-Wang algorithm for planar graphs that makes an additional 
step on the dual lattice. We prove that this process is rapidly 
mixing for the two-dimensional Ising model at all temperatures 
and for the two-dimensional Potts model for all non-critical 
temperatures (see Corollary~\ref{coro:mod}).

\section{Preliminaries} \label{sec:prelim}

In this section we introduce the models and notations 
to analyze the Swendsen-Wang algorithm. 
First we define the Potts (Ising) and the random-cluster model and 
explain their connection.

\subsection{The models}

Fix a natural number $q\ge2$ and a graph $G=(V,E)$ 
with finite vertex set 
$V$ and edge set $E\subseteq\binom{V}{2}$, 
where $\binom{V}{2}$ is the set of all subsets of $V$ with 2 
elements. Let $N:=\abs{V}$. 
As an example one can have in mind the two-dimensional square lattice 
$G_L$ of side length $L$, i.e. $V=\{1,\dots,L\}^2\subset\Z^2$ and 
$E=\{\{u,v\}\in V^2:\,\abs{u-v}=1\}$, where $\abs{\,\cdot\,}$ denotes 
the $\ell_1$ norm.

If we consider more than one graph, 
we denote by $V_G$ (resp. $E_G$) the vertex (resp. edge) set of the 
graph $G$. Furthermore we write $u\leftrightarrow v$ if 
$u$ and $v$ are \emph{neighbors} in $G$, i.e. $\{u,v\}\in E$, 
and by $\deg_G(v)$ we denote the \emph{degree} of the vertex $v$ 
in $G$.
The \emph{$q$-state Potts model} on $G$ is defined as the set of 
possible \emph{configurations} 
$\O_{\rm P}=[q]^V$, where $[q]\,{:=}\,\{1,\dots,q\}$ 
is the set of \emph{colors} (or spins),
together with the probability measure
\vspace{1mm}
\[
\pi_\beta(\sigma) \;:=\; \pi^{G}_{\beta,q}(\sigma) \;=\; 
	\frac1{Z(G,\beta,q)}\,
	\exp\left\{\beta\,\sum_{u,v:\, u\leftrightarrow v}
	\ind\bigl(\sigma(u)=\sigma(v)\bigr)\right\}
\]
for $\sigma\in\O_{\rm P}$, where $Z$ is the normalization 
constant (also called partition function) and $\beta\ge0$ is called the 
inverse temperature. This measure is called the \emph{Boltzmann} 
(or Gibbs) \emph{distribution}. 
If $q=2$ the Potts model is called the \emph{Ising model}.

A closely related model is the \emph{random cluster model} 
(also known as the FK-model), see Fortuin and Kasteleyn \cite{FK}.
It is defined on the graph $G=(V,E)$ by its state space 
$\O_{\rm RC}=\{A: A\subseteq E\}$ and the RC measure
\[
\mu_p(A) \;:=\; \mu^G_{p,q}(A) \;=\; 
\frac1{Z(G,\log(\frac1{1-p}),q)}\,
\left(\frac{p}{1-p}\right)^{\abs{A}}\,q^{C(A)},
\]
where $p\in(0,1)$, $C(A)$ is the number of connected components 
in the graph $(V,A)$, counting isolated vertices as a component, 
and $Z(\cdot,\cdot,\cdot)$ is the same normalization constant as 
for the Potts model (see \cite[Th. 1.10]{G1}).  
For a detailed introduction and related topics see \cite{G1}.

There is a tight connection between the Potts model and the 
random cluster model. Namely, if we set $p=1-e^{-\beta}$, 
we can translate a Potts configuration $\sigma\sim\pi_\beta$ to 
a random cluster state $A\sim\mu_p$ and vice versa.
To get a Potts configuration $\sigma\in\O_{\rm P}$ from 
$A\in\O_{\rm RC}$ assign a random color independently to each 
connected component of $(V,A)$. 
For the reverse way include all edges $e=\{e_1,e_2\}\in E$ 
with $\sigma(e_1)=\sigma(e_2)$ to $A$ with probability $p$. 
Hence sampling a Potts configuration according to $\pi_\beta$ is 
equivalent to sampling a RC state from $\mu_p$ if both models are 
defined on the same graph $G$ and $p=1-e^{-\beta}$.
For a proof see \cite{ES}.

\subsection{The heat-bath dynamics}

The \emph{heat bath dynamics} is probably the most studied 
Markov chain related to the Potts model (especially for $q=2$). 
This is because its mixing time is related 
to some properties of the underlying model. See e.g. 
the monograph of Martinelli \cite{M}.

The transition matrix of the 
heat-bath chain on $\O_{\rm P}$ is defined by
\begin{equation}
P_{\rm HB}(\sigma,\sigma^{v,k}) 
\;:=\; P_{{\rm HB},\beta,q}^G(\sigma,\sigma^{v,k}) \;=\;
\frac{1}{N}\,\frac{\pi_\beta(\sigma^{v,k})}
	{\sum_{l\in[q]}\pi_\beta(\sigma^{v,l})},
\label{eq:HB}
\end{equation}
where $\sigma^{v,k}(v)=k\in[q]$ and 
$\sigma^{v,k}(u)=\sigma(u)$, $u\neq v$. Otherwise 
$P_{\rm HB}(\sigma,\tau)=0$. 

This transition matrix describes the process that, at each step, 
chooses one vertex of the graph $G$ uniformly at random and changes 
only the color of this vertex with respect to the conditional 
distribution of $\pi_\beta$, given the color of all other 
vertices. It is easy to prove that this Markov chain is reversible 
with respect to $\pi_\beta$.

The spectral gap of $P_{\rm HB}$ is well known on some classes of graphs. 
For positive results see e.g. \cite{M}, \cite{BKMP}, \cite{MT} and 
\cite{LS}. 
We will state a few results in Section \ref{sec:main}.

\subsection{The Swendsen-Wang algorithm}

Now we turn to the Swendsen-Wang algorithm \cite{SW}. 
First we state the coupling of the Boltzmann distribution 
$\pi_{\beta,q}^G$ and the random-cluster measure $\mu_{p,q}^G$ 
of Edwards and Sokal \cite{ES}. 
Let us define 
\[
\O(A) \;:=\; \bigl\{\sigma\in\O_{\rm P}:\,\sigma(u)=\sigma(v)\;\;
	\forall\{u,v\}\in A\bigr\}, \quad A\in\O_{\rm RC},
\]
and
\[
E(\sigma) \;:=\; \bigl\{\{u,v\}\in E:\,\sigma(u)=\sigma(v)\bigr\},
	\quad \sigma\in\O_{\rm P}.
\]
Obviously, we have for $\sigma\in\O_{\rm P}$ and $A\subset E$ that 
$\sigma\!\in\!\O(A)\Leftrightarrow A\!\subset\!E(\sigma)$. 
Let $\sigma\in\O_{\rm P}$, $A\in\O_{\rm RC}$ and $p=1-e^{-\beta}$, 
then the joint measure of $(\sigma,A)\in\O_{\rm P}\times\O_{\rm RC}$ is
\[
\bar\mu(\sigma,A)\;:=\;\frac{1}{Z(G,\beta,q)}\,
	\left(\frac{p}{1-p}\right)^{\abs{A}}\,\ind(A\subset E(\sigma)).
\]
The marginal distributions of $\bar\mu$ are exactly $\pi_\beta$ and 
$\mu_p$, respectively.
The \emph{Swendsen-Wang algorithm} 
is based on this connection of  the random cluster and Potts models 
and performs the following two steps.
\begin{enumerate}
	\item[1)] Given a Potts configuration $\sigma_t\in\O_{\rm P}$ 
		on $G$, delete each edge of $E(\sigma_t)$ independently with 
		probability $1-p = e^{-\beta}$. This gives $A\in\O_{\rm RC}$.
	\item[2)] Assign a random color independently to each 
		connected component of $(V,A)$. Vertices of the same component 
		get the same color. This gives $\sigma_{t+1}\in\O_{\rm P}$.
\end{enumerate}
This can be seen as first choosing $A$ with respect to the 
conditional probability of $\bar\mu$ given $\sigma_t$ and then 
choosing $\sigma_{t+1}$ with respect to $\bar\mu$ given $A$.
If we define the $\abs{\O_{\rm P}}\times\abs{\O_{\rm RC}}$-matrix
\begin{equation}
T_{G,p,q}(\sigma,A) \;:=\; \frac{\bar\mu(\sigma,A)}{\bar\mu(\sigma,\O_{\rm RC})} 
\;=\; p^{\abs{A}}\,(1-p)^{\abs{E(\sigma)}-\abs{A}}\,
	\ind(A\subset E(\sigma))
\label{eq:T}
\end{equation}
and the $\abs{\O_{\rm RC}}\times\abs{\O_{\rm P}}$-matrix
\begin{equation}
T_{G,p,q}^*(A,\sigma) \;:=\; \frac{\bar\mu(\sigma,A)}{\bar\mu(\O_{\rm P},A)} 
\;=\; q^{-C(A)}\,\ind(\sigma\in\O(A)),
\label{eq:T*}
\end{equation}
then the \emph{transition matrix} of the 
Swendsen-Wang dynamics (on $\O_{\rm P}$) is defined by
\begin{equation}
P(\sigma,\tau) \;:=\; P_{\beta,q}^G(\sigma,\tau) 
\;=\; T_{G,p,q}\,T_{G,p,q}^*(\sigma,\tau), 
\qquad \sigma,\tau\in\O_{\rm P}.
\label{eq:P}
\end{equation}
Note that the transition matrix of the Swendsen-Wang dynamics 
on $\O_{\rm RC}$ is given by
\begin{equation}
\widetilde{P}(A,B) \;:=\; \widetilde{P}_{p,q}^G(A,B) 
\;=\; T_{G,p,q}^*\,T_{G,p,q}(A,B), \qquad A,B\in\O_{\rm RC}.
\label{eq:P2}
\end{equation}
\vspace{1mm}

\subsection{Spectral gap}

In the following we want to estimate the spectral gap of certain 
transition matrices of Markov chains. For an introduction to Markov 
chains and techniques to bound the convergence rate to the stationary 
distribution, see e.g. \cite{LPW}. 
If the transition matrix $P$ with state space $\Omega$ is 
ergodic, i.e. irreducible and aperiodic, and reversible with respect 
to $\pi$, i.e.
\[
\pi(x)\,P(x,y) \;=\; \pi(y)\,P(y,x) \quad \text{ for all } x,y\in\O,
\]
we know that 
$1=\xi_0>\xi_1\ge\dots\ge\xi_{\abs{\Omega}-1}>-1$, where 
the $\xi_i$ are the (real) eigenvalues of $P$. 
The \emph{spectral gap} of the Markov chain is defined by 
$\lambda(P)=1-\max\bigl\{\xi_1,\abs{\xi_{\abs{\O}-1}}\bigr\}$. 
If we are considering simultaneously a \emph{family} of graphs
  $G = (V_G,E_G)$, we say that the chain is \emph{rapidly mixing}
  for the given family if $\lambda(P)^{-1} = \mathcal{O}(|V_G|^C)$
  for some $C < \infty$.
  
The eigenvalues of the Markov chain can be expressed in terms of norms 
of the operator $P$ that maps from 
$L_2(\pi):=(\R^\O,\Vert\cdot\Vert_\pi)$ to $L_2(\pi)$, where 
scalar product and norm are given by 
$\l f,g\r_\pi=\sum_{x\in\O}f(x) g(x) \pi(x)$ and 
$\Vert f\Vert_\pi^2:=\sum_{x\in\O}f(x)^2\pi(x)$, respectively. 
The operator is defined by
\[
Pf(x) \;:=\; \sum_{y\in\O}\,P(x,y)\,f(y)
\]
and represents the expected value of the function $f$ after one step of 
the Markov chain starting in $x\in\O$. 
The \emph{operator norm} of $P$ is
\[
\Vert P\Vert_\pi \;:=\; \Vert P\Vert_{L_2(\pi)\to L_2(\pi)} 
\;=\; \max_{\Vert f\Vert_\pi\le1} \Vert Pf\Vert_\pi
\]
and we use $\Vert\cdot\Vert_\pi$ interchangeably for functions and 
operators, because it will be clear from the context which 
norm is used. It is well known that 
$\lambda(P)=1-\norm{P-S_\pi}_\pi$ for reversible $P$, 
where $S_\pi(x,y)=\pi(y)$. 
We know that reversible $P$ are self-adjoint, i.e. 
$P=P^*$, where $P^*$ is the (\emph{adjoint}) operator that 
satisfies 
$\l f, Pg\r_{\pi} \;=\; \l P^*f, g\r_{\pi}$ for all 
$f,g\in L_2(\pi)$.
Note that if we look at $T_{G,p,q}$ from \eqref{eq:T} as an 
operator mapping from $L_2(\mu_p)$ to $L_2(\pi_\beta)$, then 
$T_{G,p,q}^*$ is the adjoint operator of $T_{G,p,q}$. 
This proves that $P$ and $\widetilde{P}$ from \eqref{eq:P} and
\eqref{eq:P2} have the 
same spectral gap.

\section{Main result} \label{sec:main}

In this section we prove by comparison that the spectral gap 
of the Swendsen-Wang dynamics is bounded from below by some 
constant times the spectral gap of the heat-bath chain. 
Therefore we fix some value of $\beta\ge0$ and $q\ge2$. 

We will prove the following theorem.

\begin{theorem} \label{theorem:main}
Suppose that $P$ (resp. $P_{\rm HB}$) is the transition matrix of the 
Swendsen-Wang (resp. heat-bath) dynamics, which is reversible 
with respect to $\pi^{G}_{\beta,q}$.
Then
\[
\lambda(P) \;\ge\; c_{\rm SW}\,\lambda(P_{\rm HB}),
\]
where
\[
c_{\rm SW} \;=\; c_{\rm SW}(G,\beta,q) 
\;:=\; \frac{1}{2 q^2}\left(q\,e^{2\beta}\right)^{-4\Delta}
\]
with 
\[
\Delta \;:=\; \max_{v\in V_G}\,\deg_G(v).
\] 
\end{theorem}

By this result we get that the Swendsen-Wang dynamics is rapidly 
mixing (on graphs with bounded degree), if the usual heat-bath 
dynamics is rapidly mixing. 
Unfortunately, this result seems to be off by a factor of $N$,  
because we compare the SW dynamics with a Markov chain that changes 
only the color of one vertex of the graph per step. 
A similar comparison with the systematic scan heat-bath dynamics  
would give a better bound (see e.g. \cite{Ha}). 
But this does not seem to be possible with our techniques.

Now we will state some results that follow from Theorem 
\ref{theorem:main}.

The first corollary deals with the two-dimensional Potts model.
Recall that the two-dimensional square lattice 
$G_L$ of side length $L$ is given by $V=\{1,\dots,L\}^2\subset\Z^2$ and 
$E=\{\{u,v\}\in V^2:\,\abs{u-v}=1\}$.

Using the known lower bounds on the spectral gap of the 
heat bath dynamics on the square lattice \cite{MOS,BDC,LS}, 
we can obtain the following result:

\begin{corollary} \label{coro:2d}
Let $G_L$ be the square lattice of side-length $L$, $N=L^2$. 
Then there exist 
constants $c_{\beta}=c_{\beta}(\beta,q)>0$ and $m>0$ such that
\begin{itemize}
\item\quad $\lambda(P_{\beta,q}^{G_L})^{-1}\;\le\;c_{\beta} N$
	\qquad for $\beta<\beta_c(q)=\ln(1+\sqrt{q})$ 
\vspace{2mm}
\item\quad $\lambda(P_{\beta,q}^{G_L})^{-1}\;\le\;c\,N^m$
	\qquad for $q=2$ and $\beta=\beta_c(2)$, 
\end{itemize}
where $c^{-1}=c_{\rm SW}(G_L,\beta_c,2)$.
\end{corollary}

Not much is known about the constant $c_\beta$, but there is 
an explicit bound on the exponent $m$ in terms of an (unknown) 
crossing probability in the random cluster model, 
see \cite{LS}.

\begin{proof}[Proof of Corollary \ref{coro:2d}]
We have to show that the heat-bath dynamics is rapidly mixing 
in the cases stated above. The result for $q=2$ at the critical 
temperature is given by Lubetzky and Sly \cite{LS}.
For the high temperature result we need the recent result of 
Beffara and Duminil-Copin \cite[Th.~2]{BDC}. They prove that 
one has exponential decay of connectivity in the RC model on $\Z^2$ 
for all $q\ge1$ and 
$p<\frac{\sqrt{q}}{1+\sqrt{q}}=1-e^{-\beta_c(q)}$. 
Together with \cite[Th.~3.6]{A} and \cite[Th.~3.2]{MOS} we get 
that there exists a constant $\widetilde{c}_\beta>0$ such that
$\lambda(P_{\rm HB})^{-1}\le\widetilde{c}_\beta\,N$ on $G_L$ for 
$\beta<\beta_c(q)$. 
Thus the result follows from Theorem~\ref{theorem:main}.
\end{proof}

The next corollary relies on a result of Hayes \cite{Ha}, 
who gives a simple condition on $\beta$ for rapid mixing 
of the heat-bath dynamics for the Ising model on graphs 
of bounded degree. 
Especially, we use his result for planar graphs.

\begin{corollary} \label{coro:planar}
Let $G$ be a graph with maximum degree $\Delta$. 
Then for every $q\ge2$ and $\eps>0$ the spectral gap of 
the Swendsen-Wang dynamics satisfies
\[
\lambda(P_{\beta,q}^{G})^{-1}
	\;\le\;\frac{c}{\eps}\,N \log N,
\] 
where $c^{-1}=\frac14\, c_{\rm SW}(G,\beta,q)$, if
\[
\beta \,\le\, 2 \,\frac{1-\eps}{\Delta}
\]
or if $G$ is planar, $\Delta\ge6$ and
\[
\beta \,\le\, \frac{1-\eps}{\sqrt{3 (\Delta-3)}}.
\]
\end{corollary}

If we state the bounds on the temperature in terms of the 
RC parameter $p=1-e^{-\beta}$, this leads to the bounds 
$p\le(1-\eps)/\Delta$ in general and 
$p\le(1-\eps)/(3\sqrt{\Delta-3})$ for planar graphs.

\begin{proof}[Proof of Corollary \ref{coro:planar}]
In \cite[Prop.~14, Coro.~19]{Ha} the result is given for 
$q=2$ and $\beta\le(1-\eps)/\rho$ in terms of mixing time, 
where $\rho$ is the principal eigenvalue of the graph. 
The bounds on $\rho$ are stated therein.
Note that the different definition 
of $\pi_\beta$ causes the additional factor 2.
To generalize this result to the $q$-state Potts model one 
only has to prove an inequality like in Observation~11 in 
\cite{Ha}. 
But this is easily done, e.g. by induction over $q$.
Finally we have to bound the spectral gap in terms of the 
mixing time. For this see e.g. \cite[Th.~12.4]{LPW}.
\end{proof}

\begin{remark}
Note that one can read the statement of Theorem \ref{theorem:main} 
also as an upper bound on the spectral gap of the heat bath 
dynamics. See e.g. \cite{BCT} for a result on slow mixing of the 
Swendsen-Wang dynamics.
\end{remark}

\subsection{Proof of Theorem \ref{theorem:main}}

See Diaconis and Saloff-Coste \cite{DSC2} for an introduction 
to comparison techniques for Markov chains.
For the comparison with the Swendsen-Wang dynamics we will 
analyze the Markov chain with transition matrix 
\begin{equation}
Q=P_{\rm HB}\,P\,P_{\rm HB}.
\label{eq:Q}
\end{equation}

Since $P_{\rm HB}$ and $P$ are reversible with respect to $\pi_\beta$, 
$Q$ is also reversible.

\begin{lemma}	\label{lemma:Q}
With the definitions from above we get 
\[
\lambda(Q) \;\ge\; \lambda(P_{\rm HB}^2) \;\ge\; \lambda(P_{\rm HB}).
\]
\end{lemma}
\begin{proof}
The second inequality is trivial. For the first inequality note that, 
with $S_{\pi_\beta}(\sigma,\tau)=\pi_\beta(\tau)$ for 
$\sigma,\tau\in\O_{\rm P}$, we have 
$Q-S_{\pi_\beta}=(P_{\rm HB}-S_{\pi_\beta}) P (P_{\rm HB}-S_{\pi_\beta})$, 
which is self-adjoint. Hence we can write the spectral gap as
\[\begin{split}
\lambda(Q) \;&=\; 1- \norm{Q-S_{\pi_\beta}}_{\pi_\beta} 
\;=\;1- \norm{(P_{\rm HB}-S_{\pi_\beta}) P 
			(P_{\rm HB}-S_{\pi_\beta})}_{\pi_\beta} \\
&=\; 1- \lim_{n\to\infty}\norm{\left((P_{\rm HB}-S_{\pi_\beta}) P 
			(P_{\rm HB}-S_{\pi_\beta})\right)^n}_{\pi_\beta}^{1/n} \\
&\ge\; 1- \lim_{n\to\infty}\norm{\left((P_{\rm HB}-S_{\pi_\beta})^2 P 
			\right)^{n-1}}_{\pi_\beta}^{1/n} \\
&\ge\; 1-\norm{(P_{\rm HB}-S_{\pi_\beta})^2 P}_{\pi_\beta} 
\;\ge\; 1-\norm{(P_{\rm HB}-S_{\pi_\beta})^2}_{\pi_\beta} \\
&=\; \lambda(P_{\rm HB}^2), 
\end{split}\]
where we use $\norm{P}_{\pi_\beta}\le1$ and 
$\norm{P_{\rm HB}-S_{\pi_\beta}}_{\pi_\beta}\le1$.
\end{proof}

To prove a lower bound on the spectral gap of $P$ it remains to 
prove $\lambda(P)\ge c\lambda(Q)$ for some $c>0$.
For this we need an estimate of the transition probabilities of the 
Swendsen-Wang dynamics on $G$ with respect to some subgraph of $G$.
Therefore we prove the following lemma.

\begin{lemma} \label{lemma:G0}
Let $G=(V,E)$ be a graph and $G_0=(V,E_0)$ be a spanning 
subgraph of $G$ with $E_0\subset E$. Define $P_G := P_{\beta,q}^G$. 
Then
\vspace{2mm}
\[
c_1^{\abs{E\setminus E_0}}\,P_{G_0}(\sigma,\tau) 
\;\le\; P_G(\sigma,\tau) 
\;\le\; c_2^{\abs{E\setminus E_0}}\,P_{G_0}(\sigma,\tau)
\vspace{2mm}
\]
for all $\sigma,\tau\in\O_{\rm P}$, where
\[
c_1 \;=\; c_1(\beta) \;:=\; e^{-\beta}
\]
and
\[
c_2 \;=\; c_2(\beta,q) \;:=\; 1+q\, (e^{\beta}-1).
\]
\end{lemma}

\begin{proof}
The first inequality is already known from the proof of Lemma 3.3 in 
\cite{BCT}, but we state it here for completeness. 
Let $p=1-e^{-\beta}$, then 
\[\begin{split}
P_G(\sigma,\tau) \;&=\; \sum_{A\subset E} \,p^{\abs{A}}\, 
	(1-p)^{\abs{E(\sigma)}-\abs{A}}\, q^{-C(A)}\, 
	\ind\bigl(A\subset E(\sigma)\cap E(\tau)\bigr)\\
&\ge\; \sum_{A\subset E_0} \,p^{\abs{A}}\, 
	(1-p)^{\abs{E(\sigma)}-\abs{A}}\, q^{-C(A)}\, 
	\ind(A\subset E(\sigma)\cap E(\tau))\\
&\ge\; (1-p)^{\abs{E(\sigma)}-\abs{E_0(\sigma)}}\,P_{G_0}(\sigma,\tau)
\;\ge\; (1-p)^{\abs{E\setminus E_0}}\,P_{G_0}(\sigma,\tau).
\end{split}\]
For the second inequality suppose for now $E_0=E\setminus\{e\}$ 
for some $e\in E$ and note that $C(A\cup\{e\})\ge C(A)-1$. We get

\[\begin{split}
P_G(\sigma,\tau) \;&=\; \sum_{A\subset E(\sigma)\cap E(\tau)} 
	\,p^{\abs{A}}\, (1-p)^{\abs{E(\sigma)}-\abs{A}}\, q^{-C(A)} \\
&=\; \sum_{\substack{A\subset E(\sigma)\cap E(\tau):\\ e\in A}} 
	\,p^{\abs{A}}\, (1-p)^{\abs{E(\sigma)}-\abs{A}}\, q^{-C(A)} \\
&\qquad\quad + \sum_{\substack{A\subset E(\sigma)\cap E(\tau):\\ 
		e\notin A}} 
	\,p^{\abs{A}}\, (1-p)^{\abs{E(\sigma)}-\abs{A}}\, q^{-C(A)} \\
&\le\; \sum_{A'\subset E_0(\sigma)\cap E_0(\tau)} 
	\,p^{\abs{A'\cup\{e\}}}\, 
	(1-p)^{\abs{E(\sigma)}-\abs{A'\cup\{e\}}}\, q^{-C(A'\cup\{e\})} \\
&\qquad\quad + \sum_{A'\subset E_0(\sigma)\cap E_0(\tau)} 
	\,p^{\abs{A'}}\, (1-p)^{\abs{E(\sigma)}-\abs{A'}}\, q^{-C(A')} \\
&\le\; \frac{q\,p}{1-p}\,	\sum_{A'\subset E_0(\sigma)\cap E_0(\tau)} 
	\,p^{\abs{A'}}\, (1-p)^{\abs{E(\sigma)}-\abs{A'}}\, q^{-C(A')} \\
&\qquad\quad + \,(1-p)^{\abs{E(\sigma)}-\abs{E_0(\sigma)}}\,
	P_{G_0}(\sigma,\tau) \\
&\le\; \left(\frac{q\,p}{1-p}\, \,+\, 1\right)\,P_{G_0}(\sigma,\tau)
\;=\; \bigl(1 \,+\, q\, (e^{\beta}-1)\bigr)\,P_{G_0}(\sigma,\tau).
\end{split}\]
For $\abs{E\setminus E_0}>1$ one can iterate this technique 
$\abs{E\setminus E_0}$ times. This proves the statement.\\
\end{proof}

We use this lemma to prove that the transition probability from 
$\sigma$ to $\tau$ is similar to the probability of going from 
a neighbor of $\sigma$ to a neighbor of $\tau$. 
Recall that $\sigma^{v,k}$ is defined by $\sigma^{v,k}(v)=k\in[q]$ and 
$\sigma^{v,k}(u)=\sigma(u)$, $u\neq v$.

\begin{lemma} \label{lemma:P}
Let $\sigma,\tau\in\O_{\rm P}$, $v\in V_G$ and $k,l\in[q]$.
Then
\[
P_G(\sigma^{v,k},\tau^{v,l}) \;\le\; 
c_3^{\deg_G(v)}\;
P_G(\sigma,\tau)
\]
with
\[
c_3 \;=\; c_3(\beta,q) \;:=\; q\,e^{2\beta} - (q-1)\,e^\beta.
\]
\end{lemma}

\begin{proof}
Define $E_v:=\{e\in E_G: v\in e\}$ and $G_v:=(V,E\setminus E_v)$. 
Then, $v\in V_G$ is an isolated vertex in $G_v$. 
By the definition of the Swendsen-Wang dynamics we get that 
\[
P_{G_v}(\sigma^{v,k},\tau^{v,l}) 
\;=\; P_{G_v}(\sigma,\tau).
\]
If we set $E_0=E\setminus E_v$ we get 
$\abs{E\setminus E_0}=\deg_G(v)$ and by Lemma \ref{lemma:G0}
\[\begin{split}
P_{G}(\sigma^{v,k},\tau^{v,l}) 
\;&\le\; c_2^{\deg_G(v)}\,P_{G_v}(\sigma^{v,k},\tau^{v,l})
\;=\; c_2^{\deg_G(v)}\,P_{G_v}(\sigma,\tau) \\
&\le\;\left(\frac{c_2}{c_1}\right)^{\deg_G(v)}\,P_{G}(\sigma,\tau)\\
\end{split}\]
with $c_1$ and $c_2$ from Lemma \ref{lemma:G0}.\\
\end{proof}

Now we are able to prove the main result.

\begin{proof}[Proof of Theorem \ref{theorem:main}]
Because of Lemma \ref{lemma:Q} we only have to prove 
$\lambda(P)\ge c_{\rm SW}\,\lambda(Q)$.
Let
\[
c \;:=\; \max_{\substack{\sigma_1,\sigma_2,\tau_1,\tau_2\in\O_{\rm P}\\ 
\sigma_1\sim \sigma_2,\,\tau_1\sim \tau_2}}\;
\frac{P(\sigma_1,\tau_1)}{P(\sigma_2,\tau_2)},
\]
where $\sigma\sim\tau:\Leftrightarrow 
\sum_{v\in V}\abs{\sigma(v)-\tau(v)}\le1$. 
Note that $P_{\rm HB}(\sigma,\tau)\neq0$ if and only if 
$\sigma\sim\tau$.
We get for $\sigma_1,\tau_1\in\O_{\rm P}$ that
\[\begin{split}
Q(\sigma_1,\tau_1) 
\;&=\; \sum_{\sigma_2,\tau_2\in\O_{\rm P}}\,
P_{\rm HB}(\sigma_1,\sigma_2)\,P(\sigma_2,\tau_2)\,
	P_{\rm HB}(\tau_2,\tau_1) \\
&\le\; c\,P(\sigma_1,\tau_1) \;\sum_{\sigma_2\sim\sigma_1} 
			P_{\rm HB}(\sigma_1,\sigma_2)\,
	\sum_{\tau_2\sim\tau_1} 
			P_{\rm HB}(\tau_2,\tau_1) \\
&\le\; q\,c\,P(\sigma_1,\tau_1).
\end{split}\]
It is well-known that
\[
\lambda(P^2) \;=\; 
\min\left\{\frac{\EE_{P^2}(f)}{\Var(f)}:\,\Var(f)\neq0\right\},
\]
where $\Var(f)=\Vert f - S_{\pi_\beta} f\Vert_{\pi_\beta}^2$ and 
\[
\EE_P(f)=\frac12\sum_{\sigma,\tau}(f(\sigma)-f(\tau))^2\,
	\pi_\beta(\sigma)\,P(\sigma,\tau).
\]
We use $P^2$ instead of $P$, because this representation holds 
only if $\lambda(P)=1-\xi_1(P)$ \cite{DSC}. 
From the calculation above we 
have $\EE_{P^2}(f)\ge\frac{1}{q^2 c^2}\,\EE_{Q^2}(f)$ for every 
$f\in L_2(\pi_\beta)$ and so 
$\lambda(P^2)\ge\frac{1}{q^2 c^2}\lambda(Q^2)$. Since we have for 
reversible $P$ that $\lambda(P)\le\lambda(P^2)\le2\lambda(P)$, 
we get
\[
\lambda(P) \;\ge\; \frac{1}{2 q^2 c^2}\,\lambda(Q).
\]
It remains to bound $c$. With $c_3$ from Lemma 
\ref{lemma:P} we get for 
$\sigma_1,\sigma_2,\tau_1, \tau_2\in\O_{\rm P}$ with 
$\sigma_1\sim \sigma_2$ and $\tau_1\sim \tau_2$
\[
\frac{P(\sigma_1,\tau_1)}{P(\sigma_2,\tau_2)}
\;\le\; c_3^{\Delta}\,\frac{P(\sigma_2,\tau_1)}{P(\sigma_2,\tau_2)}
\;\le\; c_3^{2 \Delta}\,\frac{P(\sigma_2,\tau_2)}{P(\sigma_2,\tau_2)} 
\;=\; c_3^{2 \Delta}.
\]
Finally,
\[
\lambda(P) \;\ge\; \frac{1}{2 q^2 c^2}\,\lambda(Q)
\;\ge\; \frac{1}{2 q^2}\,c_3^{-4 \Delta}\,\lambda(Q)
\;\ge\; \frac{1}{2 q^2}\,(q\,e^{2\beta})^{-4 \Delta}\,\lambda(Q).
\]
This completes the proof.\\
\end{proof}

\section{The Modified Swendsen-Wang} \label{sec:ModSW}

With the techniques from the last section we do not get any positive 
result for low temperatures.
However, the Swendsen-Wang dynamics seems to be rapidly mixing also 
at low enough temperatures. Unfortunately, we were not able to prove it.

In this section we introduce a modified version of the 
Swendsen-Wang dynamics for planar graphs that is rapidly mixing 
for the two-dimensional Ising model at all temperatures. 
Roughly speaking the chain makes one step at high and 
one step at low temperatures.

\subsection{Dual graphs}

A graph $G$ is called \emph{planar} if one can embed it into $\R^2$ 
such that two edges of $G$ intersect only at a common endvertex. 
We fix such an embedding for $G$.
Then we define the \emph{dual graph} $G_D=(V_D,E_D)$ of $G$ as follows. 
Place a dual vertex in each face (including the infinite outer one) of 
the graph $G$ and connect 2 vertices by the dual edge $e_D$ if and 
only if the corresponding faces of $G$ share the boundary edge $e$ 
(see e.g. \cite[Section 8.5]{G2}). 
It is clear, that the number of 
vertices can differ in the dual graph, but we have the same number of 
edges.\\
Additionally we define a \emph{dual RC configuration} 
$A_D\subseteq E_D$ in $G_D$ to a RC state $A\subseteq E$ in 
$G$ by
\[
e\in A \;\Longleftrightarrow\; e_D\notin A_D,
\]
where $e_D$ is the edge in $E_D$ that intersects $e$ in our (fixed) 
embedding. (By construction, this edge is unique.) 
See Figure \ref{fig-dual} for the graph $G_L$ with $L=3$ and its 
dual graph $(G_L)_D$ together with 2 corresponding RC states.

\begin{figure}[ht]
\hspace{5mm}\scalebox{1}{\input{dual-graph}
\hspace*{-2cm}\raisebox{-6mm}{\input{dual-conf}}}
\vspace*{-5mm}
\caption[Dual graph and dual RC state]{Left: The graph $G_3$ (solid) 
and its dual (dashed).\\ Right: A RC state on $G_3$ (solid) and its 
dual configuration (dashed).}
\label{fig-dual}
\vspace{5mm}
\end{figure}
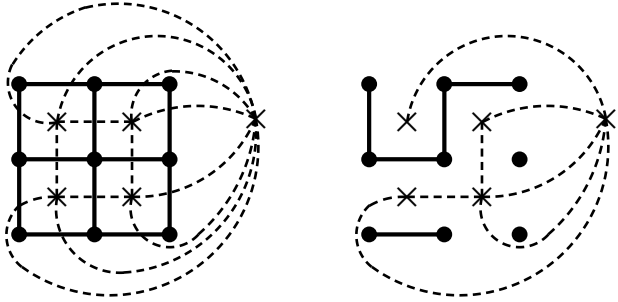

It is easy to obtain (see \cite[p.~134]{G1}) that the random 
cluster models on the (finite) 
graphs $G$ and $G_D$ are related by the equality
\begin{equation} \label{eq:dual}
\mu_{p,q}^G(A) \;=\; \mu^{G_D}_{p^*,q}(A_D),
\end{equation}
where the dual parameter $p^*$ satisfies
\begin{equation} \label{eq:dual-p}
\frac{p^*}{1-p^*} \;=\; \frac{q \,(1-p)}{p}.
\end{equation}
The self-dual point of this relation is given by 
$p_{\rm sd}(q)=\frac{\sqrt{q}}{1+\sqrt{q}}$, which 
corresponds by $p=1-e^{-\beta}$ to the critical temperature 
of the $q$-state Potts model $\beta_c(q)=\log(1+\sqrt{q})$ 
on $\Z^2$ \cite{BDC}. 
In the following we denote by $\beta^*$ the ``dual'' value 
of $\beta$, i.e. we have $p^*=1-e^{-\beta^*}$, $p=1-e^{-\beta}$ 
and \eqref{eq:dual-p}. If $G$ is fixed, we write 
$\mu^*$ instead of $\mu^{G_D}$.

\subsection{The algorithm}

To describe the transition matrix of the modified Swendsen-Wang 
algorithm, we define the (square) 
$2^{\abs{E}}\times2^{\abs{E_D}}$-matrix
\[
D(A,B) \;:=\; \ind(B=A_D), 
\quad A\subset E, B\subset E_D.
\]
By \eqref{eq:dual} we get that 
$\mu_p\cdot D = \mu^*_{p^*}$, 
respectively $\mu^*_{p^*}\cdot D^* = \mu_{p}$ 
(distributions are seen as row vectors), 
where $D^*(B,A)=D(A,B)$, $A\subset E$, $B\subset E_D$.

Now we can state the algorithm.
The transition matrix of the \emph{modified Swendsen-Wang dynamics} 
(for $\O_{\rm P}$) on the planar graph $G$ is defined by
\begin{equation}
M \;:=\; M^G_{p,q} \;=\;
T_{G,p,q}\,D\,T^*_{G_D,p^*,q}\,T_{G_D,p^*,q}\,D^*\,T^*_{G,p,q}, 
\label{eq:M}
\end{equation}
where $T_{G,p,q}$, $T^*_{G,p,q}$ and $p^*$ are defined in 
\eqref{eq:T}, \eqref{eq:T*} and \eqref{eq:dual-p}. 
This dynamics is reversible with respect to $\pi_{\beta,q}^G$.
Although this seems to be complicated, it has an easy interpretation. 
The dynamics defined by the transition matrix $M$ performs the 
following steps:
\begin{enumerate}
	\item[1)] Given a Potts configuration $\sigma_t$ on $G$, 
		generate a random 
		cluster state $A\subset E_G$ with respect to 
		$T_{G,p,q}(\sigma_t,\cdot)$.
	\item[2)] Make one step of the Swendsen-Wang dynamics 
		$\widetilde{P}_{p^*,q}^{G_D}$ starting at $A_D\subset E_{G_D}$ to 
		get a random cluster state $B_D\subset E_{G_D}$.
	\item[3)] Generate $\sigma_{t+1}$ with respect to 
		$T^*_{G,p,q}(B,\cdot)$.
\end{enumerate}

\begin{remark}
A similar, and even simpler, algorithm was considered by 
Edwards and Sokal \cite{Scom} around 1988 for the 
two-dimensional torus, which is ``self-dual'' in some sense. 
In this case they did not have the exact (planar) duality as 
given in \eqref{eq:dual}, but one can use a Metropolis 
acceptance/rejection rule to go to a ``dual configuration''. 
This algorithm could be advantageous near the critical 
temperature $\beta_c(q)$, but away from it the Metropolis 
step is rejected with probability close to 1.
\end{remark}

The next result shows that the modified Swendsen-Wang has a larger 
spectral gap than the original Swendsen-Wang on $G$ and $G_D$, 
respectively. 

\begin{proposition} \label{theorem:mod}
Let $P_{\beta,q}^G$ be the Swendsen-Wang dynamics on a planar 
graph $G$, which is 
reversible with respect to $\pi_{\beta,q}^G$, and $M$ 
as in \eqref{eq:M}. Then
\[
\lambda(M) \;\ge\; \max\Bigl\{\lambda(P_{\beta,q}^G),\,
										\lambda(P_{\beta^*,q}^{G_D})\Bigr\}.
\]
\end{proposition}

\begin{proof}
Let $\pi:=\pi_{\beta,q}^G$ and $\pi^*:=\pi_{\beta^*,q}^{G_D}$ as 
well as $\mu:=\mu_{p,q}^G$ and $\mu^*:=\mu_{p^*,q}^{G_D}$. 
Additionally, let $T_1:=T_{G,p,q}$, $T_2:=T^*_{G,p,q}$ and 
$B:=D\,T^*_{G_D,p^*,q}\,T_{G_D,p^*,q}\,D^*$. 
Then, since $M$ is reversible, for all $n\in\N$ 
(recall that $S_\pi(\sigma,\tau)=\pi(\tau)$)
\[\begin{split}
\norm{M-S_\pi}_\pi^{n+1} 
\;&=\; \norm{ T_1\,B\,T_2-S_\pi}_\pi^{n+1}
\;=\; \norm{(T_1\,B\,T_2 - S_\pi)^{n+1}}_\pi \\
\;&=\; \norm{(T_1\,B\,T_2)^{n+1} - S_\pi}_\pi 
\;=\; \norm{T_1\,(B\,T_2\,T_1)^n\,B\,T_2 - S_\pi}_\pi \\
\;&\le\; \norm{T_1}_{(\mu,\pi)}\,\norm{B}_\mu\,\norm{T_2}_{(\pi,\mu)}\,
				\norm{(B\,T_2\,T_1)^n - S_\mu}_\mu \\
\;&\le\; \norm{T_1}_{(\mu,\pi)}^2\,\norm{B}_\mu^2\,
			\norm{T_2}_{(\pi,\mu)}^2\,\norm{(T_1\,B\,T_2)^{n-1}-S_\pi}_\pi \\
\;&\le\; \norm{T_1\,B\,T_2 - S_\pi}_\pi^{n-1} 
\;=\; \norm{M-S_\pi}_\pi^{n-1} ,  
\end{split}\]
where $\norm{T_1}_{(\mu,\pi)}:=\norm{T_1}_{L_2(\mu)\to L_2(\pi)}$. 
The last inequality holds because the operator norms are all less 
than one (since $\pi T_1=\mu$, $\mu B=\mu$ and $\mu T_2=\pi$).
Taking the $n$-th root and let $n$ tends to infinity we obtain
\[\begin{split}
\norm{M - S_\pi}_\pi 
\;&=\; \lim_{n\to\infty} \norm{(B\,T_2\,T_1 - S_\mu)^n}_\mu^{1/n} 
\;\le\; \norm{B\,T_2\,T_1 - S_\mu}_\mu \\
\;&\le\; \norm{B-S_\mu}_\mu \; \norm{T_2\,T_1- S_\mu}_\mu 
\;=\; \norm{B-S_\mu}_\mu \; \norm{\widetilde{P}^G_{p,q}- S_\mu}_\mu 
\end{split}\]
with $\widetilde{P}^G_{p,q}$ from \eqref{eq:P2}. 
Since $S_{\mu}=D\,S_{\mu^*}\,D^*$ we get
\[\begin{split}
\norm{ M-S_\pi}_\pi 
\;&\le\; \norm{D(\widetilde{P}_{p^*,q}^{G_D}-S_{\mu^*})D^*}_{\mu}
\norm{\widetilde{P}^G_{p,q}- S_\mu}_\mu 
\le \norm{P_{\beta^*,q}^{G_D}-S_{\pi^*}}_{\pi^*}
\norm{P_{\beta,q}^{G}-S_{\pi}}_{\pi} \\
&\le\; \min\Bigl\{\norm{ P_{\beta^*,q}^{G_D}-S_{\pi^*}}_{\pi^*},\,
\norm{P_{\beta,q}^{G}-S_{\pi}}_{\pi}\Bigr\}.
\end{split}\]
This yields the result, because $\lambda(M)=1-\norm{ M-S_\pi}_\pi$.
\end{proof}

\subsection{Examples}

Now we present lower bounds on the spectral gap of 
the modified Swendsen-Wang dynamics for two examples. 
The first one is the Potts model on trees, which is 
mainly a toy example, because one knows that the 
original Swendsen-Wang is rapidly mixing at all 
temperatures \cite{CF}. We present this example to 
show that the modified Swendsen-Wang can improve the 
original one at all temperatures.

The second example is the Potts model on the square lattice, 
where we prove rapid mixing for all non-critical temperatures.

\subsubsection{Potts model on trees}

Let the graph $G$ be a tree. 
By definition, $G$ has only one face: 
the outer one. Hence, the dual graph $G_D$ has only one vertex with 
$\abs{E_G}$ loops. Therefore the Swendsen-Wang dynamics 
$\widetilde{P}$ on $G_D$ is simply independent bond percolation.
It follows $\lambda(\widetilde{P})=1$ and by 
Proposition~\ref{theorem:mod} we get

\begin{corollary}
Let $G$ be a tree and $M$ be the transition matrix of the modified 
Swendsen-Wang dynamics as in \eqref{eq:M}. Then we get for all 
$\beta\ge0$ and $q\in\N$ that
\[
\lambda(M) \;=\; 1.
\]
\end{corollary}

\subsubsection{The Potts model on the square lattice}

Now we consider the Potts model on the two-dimensional square 
lattice of side length $L$. This is the graph $G=G_L=(V,E)$ with 
$V=\{1,\dots,L\}^2$ and $E=\{\{u,v\}\in V^2:\,\abs{u-v}=1\}$.
For the definition of the dual graph we use the following 
notation of the union of graphs.
Let $G=(V,E)$, $W\subset V$ and $v^*\notin V$ be an additional 
auxiliary vertex. Then we denote by $G' =G\cup_W v^*$ the graph 
$G'=(V',E')$ with $V'=V\cup \{v^*\}$ and 
$E'=E\cup\{\{v^*,u\}:\,u\in W\}$. 
Furthermore we denote the ``boundary'' of $G_L$ by 
$\delta_L=\bigl\{(v_1,v_2)\in \Z^2:\,
v_1\in\{1,L\} \text{ or } v_2\in\{1,L\}\bigr\}\subset V_{G_L}$.\\
By the construction of the dual graph it is easy to obtain 
(see Figure~\ref{fig-dual}) that 
the dual graph of the square lattice $G=G_L$ is given by 
\[
G_D \;=\; G_{L-1}\,\cup_{\delta_{L-1}}\,v^*.
\]
We cannot apply Theorem~\ref{theorem:main} directly for this graph, 
because it has no bounded degree, i.e. 
$\deg_{G_D}(v^*)=\abs{\delta_{L-1}}=4(L-1)$. Hence, we have to 
prove the following modification of Theorem \ref{theorem:main}.

\begin{cit}[Theorem~\ref{theorem:main}$'$]\label{theorem:main2}
Let $P$ be the transition matrix of the Swendsen-Wang dynamics, 
which is reversible with respect to $\pi^G_{\beta,q}$. 
Furthermore let $v\in~V_G$ be any vertex, $k\in[q]$,
\[
\Lambda_v^k \;:=\; \{\sigma\in\O_P:\,\sigma(v)=k\}
\]
and
\begin{equation}
P_{\Lambda_v^k}(\sigma,\sigma^{u,l}) \;=\; \frac{1}{N-1}\,
\frac{\pi_\beta(\sigma^{u,l}\;\rule[-1.5mm]{0.2mm}{5mm}\;\Lambda_v^k)}
	{\sum_{j\in[q]}\pi_\beta(\sigma^{u,j}\;\rule[-1.5mm]{0.2mm}{5mm}\;\Lambda_v^k)},
\label{eq:HB2}
\end{equation}
for $\sigma,\tau\in\Lambda_v^k$ with $\abs{\sigma-\tau}=1$.\\
Then
\[
\lambda(P) \;\ge\; \widetilde{c}_{\rm SW}\,\lambda(P_{\Lambda_v^k}^2),
\]
where
\[
\widetilde{c}_{\rm SW} \;=\; \widetilde{c}_{\rm SW}(G,\beta,q) 
\;:=\; \frac1{2 q^2}\left(q\,e^{2\beta}\right)^{-4\widetilde{\Delta}}
\]
with 
\[
\widetilde{\Delta} \;:=\; \max_{u\in V_G\setminus\{v\}}\,\deg_G(u).
\] 
\end{cit}

\begin{proof}
The proof of this theorem is very similar to the proof of 
Theorem \ref{theorem:main}.
First we define the following ``flip'' 
transition matrices. Let $\sigma\in\O_{\rm P}$ and 
$\tau\in\Lambda_v^k$, then
\[
F_1(\sigma,\tau) \;:=\; \ind(\tau\,=\,\sigma-\sigma(v)+k)
\]
and
\[
F_2(\sigma,\tau) \;:=\; F_1^*(\sigma,\tau) \;=\;
\frac{1}{q}\,\sum_{l=0}^{q-1}\,\ind(\tau=\sigma+l),
\]
where $(\sigma+l)(u)\,:=\,(\sigma(u)+l-1\mod q)+1$ for 
$\sigma\in\O_{\rm P}$ and $l\in\Z$. 
Let $\pi:=\pi^G_{\beta,q}$ and $\widetilde{\pi}:=
\pi_{\beta,q}^G(\cdot\,\rule[-1.5mm]{0.2mm}{5mm}\,\Lambda_v^k)$.
It is easy to check that $\pi F_1=\widetilde{\pi}$ and 
$\widetilde{\pi}F_2=\pi$.
Following the same ideas as in Section \ref{sec:main} 
with
\[
\widetilde{Q} \;=\; F_1\,P_{\Lambda_v^k}\,F_2\,P\,
	F_1\,P_{\Lambda_v^k}\,F_2,
\]
which is reversible with respect to $\pi=\pi_{\beta,q}^G$, we get 
\[
\norm{ \widetilde{Q}-S_\pi}_\pi \;\le\; 
\norm{ P_{\Lambda_v^k}-S_{\widetilde{\pi}}}_{\widetilde{\pi}}^2
\;=\; \norm{ P_{\Lambda_v^k}^2-S_{\widetilde{\pi}}}_{\widetilde{\pi}}.
\]
Similarly as Lemma \ref{lemma:Q}, 
$\lambda(\widetilde{Q})\ge\lambda(P_{\Lambda_v^k}^2)$. 
It remains to prove that 
$\lambda(P)\ge\widetilde{c}_{\rm SW}\,\lambda(\widetilde{Q})$. 
For this we define for $\sigma\in\O_{\rm P}$, 
$\sigma^k:=\sigma-\sigma(v)+k\in\Lambda_v^k$ and 
$\sigma\equiv\tau :\Leftrightarrow \{\exists l\in[q]: \tau=\sigma^l\}$.
By the construction of the Swendsen-Wang dynamics we have for 
$k,l\in[q]$ that
\[
P(\sigma^k,\tau^l) \;=\; P(\sigma,\tau) \qquad 
	\forall\sigma,\tau\in\O_{\rm P}.
\]
Hence we get for $\sigma,\tau\in\O_{\rm P}$ with 
\[
\widetilde{c} \;:=\; 
\max_{\substack{\sigma_1,\sigma_2,\tau_1,\tau_2\in\Lambda_v^k\\ 
\sigma_1\sim \sigma_2,\,\tau_1\sim \tau_2}}\;
\frac{P(\sigma_1,\tau_1)}{P(\sigma_2,\tau_2)}
\]
that
\[\begin{split}
\widetilde{Q}(\sigma,\tau) 
\;&=\; \sum_{\substack{\sigma_1,\tau_1\in\Lambda_v^k,\\
								\sigma_2,\tau_2\in\O_{\rm P}}}
	P_{\Lambda_v^k}(\sigma^k,\sigma_1)\,F_2(\sigma_1,\sigma_2)\,
	P(\sigma_2,\tau_2)\,P_{\Lambda_v^k}(\tau_2^k,\tau_1)\,
	F_2(\tau_1,\tau) \\
\;&=\; \sum_{\substack{\sigma_1,\tau_1\in\Lambda_v^k,\\
								\tau_2\in\O_{\rm P}}}
	P_{\Lambda_v^k}(\sigma^k,\sigma_1)\,
	P(\sigma_1,\tau_2^k)\,P_{\Lambda_v^k}(\tau_2^k,\tau_1)\,
	F_2(\tau_1,\tau) \\
&\le\; \widetilde{c}\,P(\sigma^k,\tau) \;
	\sum_{\sigma_1\sim\sigma^k} P_{\Lambda_v^k}(\sigma^k,\sigma_1)\,
	\sum_{\tau_1\equiv\tau} F_2(\tau_1,\tau) \,
	\sum_{\tau_2:\tau_2^k\sim\tau_1} P_{\Lambda_v^k}(\tau_2^k,\tau_1) \\
&=\; \widetilde{c}\,P(\sigma,\tau) \;
	\sum_{\tau_1\equiv\tau} F_2(\tau_1,\tau) \,
	\sum_{\tau_2:\tau_2^k\sim\tau_1} P_{\Lambda_v^k}(\tau_2^k,\tau_1).
\end{split}\]
But $\tau_1$ is unique, because $\tau_1\equiv\tau$ and 
$\tau_1\in\Lambda_v^k$. Therefore,
\[\begin{split}
\widetilde{Q}(\sigma,\tau) 
\;&\le\; \widetilde{c}\,P(\sigma,\tau) \; \frac{1}{q} \,
	\sum_{\tau_2:\tau_2^k\sim\tau^k} P_{\Lambda_v^k}(\tau_2^k,\tau_1) \\
\;&=\; \widetilde{c}\,P(\sigma,\tau) \; \frac{1}{q} \,
	\sum_{l=1}^q\,\sum_{\tau_2\in\Lambda_v^l:\tau_2^k\sim\tau^k} 
	P_{\Lambda_v^k}(\tau_2^k,\tau_1) \\
\;&\le\; q\,\widetilde{c}\,P(\sigma,\tau).
\end{split}\]
With $c_3$ from Lemma~\ref{lemma:P} we get for 
$\sigma_1,\sigma_2,\tau_1,\tau_2\in\Lambda_v^k$ with 
$\sigma_1\sim \sigma_2$ and $\tau_1\sim \tau_2$ 
(since $\sigma_1(v)=\sigma_2(v)$ and $\tau_1(v)=\tau_2(v)$) that
\[
\frac{P(\sigma_1,\tau_1)}{P(\sigma_2,\tau_2)}
\;\le\; c_3^{\widetilde{\Delta}}\,
	\frac{P(\sigma_2,\tau_1)}{P(\sigma_2,\tau_2)}
\;\le\; c_3^{2 \widetilde{\Delta}}\,
	\frac{P(\sigma_2,\tau_2)}{P(\sigma_2,\tau_2)} 
\;=\; c_3^{2 \widetilde{\Delta}}.
\]
By the same ideas as in the proof of Theorem~\ref{theorem:main} 
we conclude
\[
\lambda(P) \;\ge\; \frac{1}{2 q^2 \widetilde{c}^2}\,\lambda(Q)
\;\ge\; \frac{1}{2 q^2}\,c_3^{-4 \widetilde{\Delta}}\,\lambda(Q)
\;\ge\; \frac{1}{2 q^2}\,(q\,e^{2\beta})^{-4 \widetilde{\Delta}}\,
	\lambda(\widetilde{Q}).
\]
This completes the proof.\\
\end{proof}

Now we can apply Theorem~\ref{theorem:main}$'$ for 
$G_D \;=\; G_{L-1}\,\cup_{\delta_{L-1}}\,v^*$ with $v=v^*$, because 
then $\widetilde{\Delta}=4$. For the result on the spectral gap 
of the Markov chain with transition matrix $M$ we need a lower bound 
on the spectral gap of $P_{\Lambda_{v^*}^k}$ on $G_D$.
But this is simply the heat-bath dynamics on $G_{L-1}$ with 
constant (i.e. $k$) boundary condition. From the well-known 
results on the spectral gap of the heat-bath dynamics on the square 
lattice at high temperatures (see \cite{MO1} or \cite{MOS}) we know 
that the $\mathcal{O}(N)$ bound on the inverse-gap also holds for 
arbitrary boundary conditions. Hence, we get the following.

\begin{corollary} \label{coro:mod}
Let $G_L$ be the square lattice of side-length $L$, $N=L^2$
and $M=M^{G_L}_{p,q}$ as in \eqref{eq:M}. Then there exist 
constants $c_\beta>0$ and $m>0$ such that
\begin{itemize}
\item\quad $\lambda(M)^{-1}\;\le\;c_\beta\,N$
	\qquad for $\beta\neq\beta_c(q)=\ln(1+\sqrt{q})$ 
\vspace{2mm}
\item\quad $\lambda(M)^{-1}\;\le\;c\,N^m$
	\qquad for $q=2$ and $\beta=\beta_c(2)$, 
\end{itemize}
where $c=c_{\rm SW}(G_L,\beta_c,2)$.
\end{corollary}

\vspace{1cm}
\subsection*{Acknowledgements}
The author thanks Alan Sokal for carefully reading the paper 
and for his valuable comments. I am also very grateful to 
my colleages Erich Novak, Daniel Rudolf and Aicke Hinrichs for 
the numerous discussions.

{
\linespread{1}
\bibliographystyle{amsalpha}
\bibliography{Bibliography}
}

\end{document}

%% file: dual-graph.tex
\psset{xunit=1.0cm,yunit=1.0cm,algebraic=true,dotstyle=*,dotsize=3pt 0,linewidth=0.8pt,arrowsize=3pt 2,arrowinset=0.25}
\begin{pspicture*}(-0.43,-0.06)(6.02,4.43)
\psline[linewidth=1.6pt](1,3)(2,3)
\psline[linewidth=1.6pt](2,3)(3,3)
\psline[linewidth=1.6pt](3,3)(3,2)
\psline[linewidth=1.6pt](3,2)(3,1)
\psline[linewidth=1.6pt](3,1)(2,1)
\psline[linewidth=1.6pt](2,1)(2,2)
\psline[linewidth=1.6pt](2,2)(3,2)
\psline[linewidth=1.6pt](2,2)(2,3)
\psline[linewidth=1.6pt](1,3)(1,2)
\psline[linewidth=1.6pt](1,2)(2,2)
\psline[linewidth=1.6pt](2,1)(1,1)
\psline[linewidth=1.6pt](1,1)(1,2)
\psline[linewidth=1.0pt,linestyle=dashed,dash=3pt 2pt](1.5,1.5)(1.5,2.5)
\psline[linewidth=1.0pt,linestyle=dashed,dash=3pt 2pt](1.5,2.5)(2.5,2.5)
\psline[linewidth=1.0pt,linestyle=dashed,dash=3pt 2pt](2.5,2.5)(2.5,1.5)
\psline[linewidth=1.0pt,linestyle=dashed,dash=3pt 2pt](2.5,1.5)(1.5,1.5)
\parametricplot[linewidth=1.0pt,linestyle=dashed,dash=3pt 2pt]{0.1852913062710569}{2.989280090214289}{1*1.34*cos(t)+0*1.34*sin(t)+2.83|0*1.34*cos(t)+1*1.34*sin(t)+2.3}
\parametricplot[linewidth=1.0pt,linestyle=dashed,dash=3pt 2pt]{1.1350083349985625}{2.0596018113318753}{1*1.84*cos(t)+0*1.84*sin(t)+3.37|0*1.84*cos(t)+1*1.84*sin(t)+0.87}
\parametricplot[linewidth=1.0pt,linestyle=dashed,dash=3pt 2pt]{4.6717617082515}{5.883637174045906}{1*1.71*cos(t)+0*1.71*sin(t)+2.57|0*1.71*cos(t)+1*1.71*sin(t)+3.21}
\parametricplot[linewidth=1.0pt,linestyle=dashed,dash=3pt 2pt]{1.5028552028831879}{2.100187227182034}{1*0.86*cos(t)+0*0.86*sin(t)+1.44|0*0.86*cos(t)+1*0.86*sin(t)+0.64}
\parametricplot[linewidth=1.0pt,linestyle=dashed,dash=3pt 2pt]{2.311045279311494}{4.0371577644321395}{1*0.54*cos(t)+0*0.54*sin(t)+1.37|0*0.54*cos(t)+1*0.54*sin(t)+0.99}
\parametricplot[linewidth=1.0pt,linestyle=dashed,dash=3pt 2pt]{-2.1993684311435224}{0.19067988888988038}{1*1.98*cos(t)+0*1.98*sin(t)+2.2|0*1.98*cos(t)+1*1.98*sin(t)+2.17}
\parametricplot[linewidth=1.0pt,linestyle=dashed,dash=3pt 2pt]{2.8453007051607155}{5.540388001593068}{1*0.52*cos(t)+0*0.52*sin(t)+3|0*0.52*cos(t)+1*0.52*sin(t)+1.35}
\parametricplot[linewidth=1.0pt,linestyle=dashed,dash=3pt 2pt]{5.428195851138397}{6.21778798060977}{1*2.24*cos(t)+0*2.24*sin(t)+1.91|0*2.24*cos(t)+1*2.24*sin(t)+2.69}
\parametricplot[linewidth=1.0pt,linestyle=dashed,dash=3pt 2pt]{1.651943081227404}{3.288962910742752}{1*0.59*cos(t)+0*0.59*sin(t)+3.08|0*0.59*cos(t)+1*0.59*sin(t)+2.59}
\parametricplot[linewidth=1.0pt,linestyle=dashed,dash=3pt 2pt]{0.523108668073518}{1.593136560077821}{1*1.25*cos(t)+0*1.25*sin(t)+3.06|0*1.25*cos(t)+1*1.25*sin(t)+1.92}
\parametricplot[linewidth=1.0pt,linestyle=dashed,dash=3pt 2pt]{2.956350314060717}{4.7749714687606275}{1*0.85*cos(t)+0*0.85*sin(t)+2.34|0*0.85*cos(t)+1*0.85*sin(t)+1.34}
\parametricplot[linewidth=1.0pt,linestyle=dashed,dash=3pt 2pt]{-1.5125031955649737}{0.0989999773941644}{1*1.87*cos(t)+0*1.87*sin(t)+2.28|0*1.87*cos(t)+1*1.87*sin(t)+2.36}
\parametricplot[linewidth=1.0pt,linestyle=dashed,dash=3pt 2pt]{2.677755480453707}{4.9277297505381314}{1*0.54*cos(t)+0*0.54*sin(t)+1.39|0*0.54*cos(t)+1*0.54*sin(t)+3.02}
\parametricplot[linewidth=1.0pt,linestyle=dashed,dash=3pt 2pt]{1.8843071408090328}{2.583159724740975}{1*1.8*cos(t)+0*1.8*sin(t)+2.44|0*1.8*cos(t)+1*1.8*sin(t)+2.31}
\parametricplot[linewidth=1.0pt,linestyle=dashed,dash=3pt 2pt]{0.17396397227159288}{1.8107485743137344}{1*1.85*cos(t)+0*1.85*sin(t)+2.32|0*1.85*cos(t)+1*1.85*sin(t)+2.22}
\psdots[dotsize=6pt 0](1,3)
\psdots[dotsize=6pt 0](1,2)
\psdots[dotsize=6pt 0](1,1)
\psdots[dotsize=6pt 0](2,3)
\psdots[dotsize=6pt 0](2,2)
\psdots[dotsize=6pt 0](2,1)
\psdots[dotsize=6pt 0](3,3)
\psdots[dotsize=6pt 0](3,2)
\psdots[dotsize=6pt 0](3,1)
\psdots[dotsize=8pt 0,dotstyle=x](1.5,1.5)
\psdots[dotsize=8pt 0,dotstyle=x](2.5,1.5)
\psdots[dotsize=8pt 0,dotstyle=x](1.5,2.5)
\psdots[dotsize=8pt 0,dotstyle=x](2.5,2.5)
\psdots[dotsize=8pt 0,dotstyle=x](4.15,2.54)
\end{pspicture*}

%% file: dual-conf.tex
\psset{xunit=1.0cm,yunit=1.0cm,algebraic=true,dotstyle=*,dotsize=3pt 0,linewidth=0.8pt,arrowsize=3pt 2,arrowinset=0.25}
\begin{pspicture*}(-0.36,-0.66)(6.08,3.82)
\psline[linewidth=1.6pt](2,3)(3,3)
\psline[linewidth=1.6pt](2,2)(2,3)
\psline[linewidth=1.6pt](1,3)(1,2)
\psline[linewidth=1.6pt](1,2)(2,2)
\psline[linewidth=1.6pt](2,1)(1,1)
\psline[linewidth=1.0pt,linestyle=dashed,dash=3pt 2pt](2.5,2.5)(2.5,1.5)
\psline[linewidth=1.0pt,linestyle=dashed,dash=3pt 2pt](2.5,1.5)(1.5,1.5)
\parametricplot[linewidth=1.0pt,linestyle=dashed,dash=3pt 2pt]{0.1852913062710569}{2.989280090214289}{1*1.34*cos(t)+0*1.34*sin(t)+2.83|0*1.34*cos(t)+1*1.34*sin(t)+2.3}
\parametricplot[linewidth=1.0pt,linestyle=dashed,dash=3pt 2pt]{1.1350083349985625}{2.0596018113318753}{1*1.84*cos(t)+0*1.84*sin(t)+3.37|0*1.84*cos(t)+1*1.84*sin(t)+0.87}
\parametricplot[linewidth=1.0pt,linestyle=dashed,dash=3pt 2pt]{4.6717617082515}{5.883637174045906}{1*1.71*cos(t)+0*1.71*sin(t)+2.57|0*1.71*cos(t)+1*1.71*sin(t)+3.21}
\parametricplot[linewidth=1.0pt,linestyle=dashed,dash=3pt 2pt]{1.5028552028831879}{2.100187227182034}{1*0.86*cos(t)+0*0.86*sin(t)+1.44|0*0.86*cos(t)+1*0.86*sin(t)+0.64}
\parametricplot[linewidth=1.0pt,linestyle=dashed,dash=3pt 2pt]{2.311045279311494}{4.0371577644321395}{1*0.54*cos(t)+0*0.54*sin(t)+1.37|0*0.54*cos(t)+1*0.54*sin(t)+0.99}
\parametricplot[linewidth=1.0pt,linestyle=dashed,dash=3pt 2pt]{-2.1993684311435224}{0.19067988888988038}{1*1.98*cos(t)+0*1.98*sin(t)+2.2|0*1.98*cos(t)+1*1.98*sin(t)+2.17}
\parametricplot[linewidth=1.0pt,linestyle=dashed,dash=3pt 2pt]{2.8453007051607155}{5.540388001593068}{1*0.52*cos(t)+0*0.52*sin(t)+3|0*0.52*cos(t)+1*0.52*sin(t)+1.35}
\parametricplot[linewidth=1.0pt,linestyle=dashed,dash=3pt 2pt]{5.428195851138397}{6.21778798060977}{1*2.24*cos(t)+0*2.24*sin(t)+1.91|0*2.24*cos(t)+1*2.24*sin(t)+2.69}
\psdots[dotsize=6pt 0](1,3)
\psdots[dotsize=6pt 0](1,2)
\psdots[dotsize=6pt 0](1,1)
\psdots[dotsize=6pt 0](2,3)
\psdots[dotsize=6pt 0](2,2)
\psdots[dotsize=6pt 0](2,1)
\psdots[dotsize=6pt 0](3,3)
\psdots[dotsize=6pt 0](3,2)
\psdots[dotsize=6pt 0](3,1)
\psdots[dotsize=8pt 0,dotstyle=x](1.5,1.5)
\psdots[dotsize=8pt 0,dotstyle=x](2.5,1.5)
\psdots[dotsize=8pt 0,dotstyle=x](1.5,2.5)
\psdots[dotsize=8pt 0,dotstyle=x](2.5,2.5)
\psdots[dotsize=8pt 0,dotstyle=x](4.15,2.54)
\end{pspicture*}